\documentclass[12pt]{amsart}
\usepackage[margin=30mm]{geometry}

\usepackage{ucs}
\usepackage[latin1,utf8]{inputenc}

\newtheorem{theorem}{Theorem}[section]
\newtheorem{lemma}[theorem]{Lemma}
\newtheorem{proposition}[theorem]{Proposition}
\newtheorem{cor}[theorem]{Corollary}
\newtheorem{rem}[theorem]{Remark}
\newtheorem{remark}[theorem]{Remark}

\numberwithin{equation}{section}

\newcommand{\ze}{\mathbb{Z}}
\newcommand{\re}{\mathbb{R}}

\newcommand{\nx}{\textbf{x}}

\newcommand{\ns}{\textbf{s}}

\newcommand{\ny}{\textbf{y}}
\newcommand{\nS}{\textbf{S}}
\newcommand{\nT}{\textbf{T}}
\newcommand{\nL}{\textbf{L}}
\newcommand{\nBr}{\textbf{Br}}

\newcommand{\nm}{\textbf{m}}
\newcommand{\na}{\textbf{a}}

\newcommand{\oZ}{\overline{Z}}

\usepackage{hyperref}
\hypersetup
{
    pdfauthor={Gregorio Moreno Flores},
    pdfsubject={Asymmetric directed polymers in a random environment},
    pdftitle={Asymmetric directed polymers in a random environment},
    pdfkeywords={Directed polymers in random environment, Directed polymers in Brownian environments,
Queues in tandem, Random matrices}
}

\title[Asymmetric directed polymers]{Asymmetric directed polymers in random environments}

\author{Gregorio R. MORENO FLORES}
\address{\noindent Department of Mathematics \\ University of Wisconsin-Madison \\ Van Vleck Hall, 480 Linden Dr. \\
Madison, Wisconsin 53706\\ USA
}

\email{gregorio.random@gmail.com}

\thanks{Partially supported by Beca Conicyt-Ambassade de France and CNRS, UMR $7599$ "Probabilit\'es et
Mod\`eles Al\'eatoires".}

\date{}

\begin{document}

\begin{abstract}
The model of Brownian Percolation has been introduced as an approximation of discrete last-passage
percolation models close to the axis. It allowed to compute some explicit limits and prove fluctuation
theorems for these, based on the relations between the Brownian percolation and random matrices. 

Here, we present two approaches that allow to treat discrete asymmetric models of directed polymers. In both cases, the behaviour is universal, meaning that the
results do not depend on the precise law of the environment as long as it satisfies some natural moment assumptions.

First, we establish an approximation of asymmetric discrete directed polymers in random environments at very high temperature
by a continuous-time directed polymers model in a Brownian environment, much in the same way than the last passage percolation case.
The key ingredient is a strong embedding argument developed by K\'omlos, Major and T\'usnady.

Then, we study the partition function of a $1+1$-dimensional directed polymer in a random environment with a drift tending to infinity. We 
give an  explicit expression for the free energy based on known asymptotics for last-passage percolation and compute the order of the fluctuations of 
the partition function. We conjecture that the law of the properly rescaled fluctuations converges to the GUE Tracy-Widom distribution.
\end{abstract}

\maketitle

\tableofcontents

\section{Introduction}

The Brownian Percolation model was introduced by Glynn and Whitt in \cite{GW}, where the authors studied the asymptotic 
of passage times for customers in an infinite network of M/M/1 queues in tandem. This continuous model was easier to handle
than the original discrete problem, mostly because of the scaling properties of the Brownian motion.

Let state the problem more precisely in its original setting:
Let $\Omega_{N, M}$ be the set of directed paths from $(0,0)$ to $(N,M)$, i.e., the paths with steps equal to $(0,1)$ or
$(1,0)$. 
Let $\{ \eta(x): \, x \in \ze^2\}$ be a collection of (centered) i.i.d. random variables with
finite exponential moments $e^{\lambda(\beta)}= Q(e^{\beta \eta}) < +\infty$, which
will be referred as the environment variables, or just as the  environment.

Define 

\begin{eqnarray}\label{Mdiscrete}
 T(N,M) = \max_{\ns \in \Omega_{N,M}} H(\nS).
\end{eqnarray}

\noindent where $H({\bf S})= \sum_{(t,x)\in \nS} \eta(t,x)$ will be called the energy of the path $\nS$. This is usually referred to as a 
last-passage percolation problem (LPP). It can be interpreted as the departure time of the $M$-th customer from the $N$-th queue in a series of 
queues in tandem. The
variable $\eta(k,n)$ has then to be understood as the service time of the $k$-th customer in the $n$-th queue.

A regime of special interest occurs when 

$$M=O(N^a)$$ 

\noindent for some $a\in (0,1)$. Glynn and Whitt \cite{GW} proved that

\begin{eqnarray}\label{GW-limit}
 \lim_{N\to +\infty} \frac{T(N,\lfloor xN^a \rfloor)}{N^{(1+a)/2}} = c\sqrt{x},
\end{eqnarray}

\noindent where the constant is independent of $a$ and of the distribution of the service times, given that they satisfy some
mild integrability conditions. The proof used
a strong approximation of sums of i.i.d. random variables by Brownian motions (see \cite{KMT1, KMT2}) in order to approximate
$T(N,\lfloor xN^a \rfloor)$ by the corresponding maximal energy along continuous-time paths in a Brownian environment 
(see below for precise definitions). Then, scaling arguments lead to $(\ref{GW-limit})$.
Based on simulations, they conjectured that $c=2$. 

\noindent The proof of this conjecture was first given by Sepp\"al\"ainen in \cite{Sepp}. It uses a coupling
between queues in tandem and TASEP. Later proofs used an interesting relation between the Brownian model
and the eigenvalues of random matrices. For a shorter proof using ideas from queueing theory and Gaussian concentration, 
see \cite{HMO}. 
A complete review of the ideas of these proofs can be found in \cite{Mthesis}.

Let us now state the following as a summary of the previous discussion:

\begin{theorem}\label{2LPP}
 \begin{eqnarray}\label{GW-limit-2}
 \lim_{N\to +\infty} \frac{T(N,\lfloor xN^a \rfloor)}{N^{(1+a)/2}} = 2\sqrt{x},
\end{eqnarray}

\noindent in probability.
\end{theorem}

\noindent Some fluctuation results are also available (see \cite{BSu,BM}). The limiting law is identified
as the Tracy-Widom distribution. This is closely related to the link between Brownian percolation and random
matrices we have  mentioned. See also \cite{I} for large deviations 
results at the Tracy-Widom scale. As usual in this type of models, the upper deviations are much larger 
than the lower ones (see \cite{J-shape} for last-passage percolation, \cite{DZ} and \cite{Sepp-LD} for the related
model of increasing subsequences in the plane and \cite{Ben} for directed polymers. See also \cite{Ledoux} for a general
discussion on the subject, including random matrices). 
This can be explained heuristically by noticing that, in order to increase the values 
of the $\max$, it is enough to increase the values of the environment along a single path. Decreasing the 
value of the $\max$ requires to decrease the values of the whole environment. 

\vspace{2ex}

We will be mostly concerned with non-zero temperature analogs to the LPP problem, namely directed polymers in random environment.
Let $P_{N,M}$ be the uniform probability measure on $\Omega_{N,M}$. 
For a given realization of the environment, we define on $\Omega_{N,M}$ the polymer measure
at inverse temperature $\beta$ as

\begin{eqnarray}\label{polymer}
 \mu^{\beta}_{N,M}(\omega = {\bf S}) = 
\frac{1}{Z_{\beta}(N,M)} e^{\beta H({\bf S})} P_{N,M}(\omega = {\bf S}), \quad    \forall \, {\bf S}\in\Omega_{N,M},
\end{eqnarray}

\noindent where  $Z_{\beta}(N,M)$ is a normalizing constant
called the (point-to-point) partition function, given by

\begin{eqnarray}
 Z_{\beta}(N,M) =  P_{N,M} \left(e^{\beta H(\omega)} \right).
\end{eqnarray}

It is easy to show the existence of the limit of the free energy in the regime considered above for the LPP. Indeed,
for $M=O(N^a)$ for some $a\in (0,1)$, the following limit holds for almost every realization of 
environment:

\begin{eqnarray}\label{thermo-thbox}
 \lim_{N\to +\infty} \frac{1}{N^{(1+a)/2}} \log Z_{\beta}(N,N^a) = 2 \beta.
\end{eqnarray}

\noindent The proof is straightforward as it applies directly the corresponding result for last-passage
percolation. Just note that

\begin{eqnarray}\label{sandwich}
-\log |\Omega_{N,N^a}| +\beta T(N,N^a) \leq \log Z_{\beta}(N,N^a) \leq  \beta T(N,N^a),
\end{eqnarray}

\noindent observe that $\log |\Omega_{N,N^a}| = O(N^a \log N)$, divide by $N^{(1+a)/2}$ and let $N$ goes
to $+\infty$. 

\smallskip

To obtain a non trivial regime, we have to ensure that the normalizing term is of the same order than 
$|\Omega_{N,N^a}|$. This will be done by increasing the temperature with $N$ (equivalently, decreasing $\beta$). 
Although this is not the usual situation in statistical mechanics, it allows us to recover a well known model 
of continuous-time directed polymer in a Brownian environment (see below for a precise definition). Until now, no 
precise relation between discrete models and this Brownian model has been given in the literature. 

\vspace{2ex}

Let us introduce more precisely the Brownian setting: let $(B^{(i)}_{\cdot})_i$ be an i.i.d. sequence of one-dimensional 
Brownian motions. Let $\Omega^c_{N,M}$ be the set of increasing sequences $0=u_0 < u_1 < \cdots
< u_{M} < u_{M+1}=N$. This can be identified as the set of piecewise constant paths with $M$ positive jumps of size $1$ 
in the interval $[0,N]$. Note that $|\Omega^c_{N,M}|= N^M / M!$, where $|\cdot|$ stands here for the Lebesgue measure. 
Denote by $P^c_{N,M}$ the uniform probability measure on $\Omega^c_{N,M}$.
For $\textbf{u} \in \Omega^c_{M,N}$, define

\begin{eqnarray}\label{Br}
 {\bf Br}(N,M)({\bf u}) = \nBr({\bf u}) = \sum^{M}_{i=0} (B^{(i)}_{u_{i+1}}-B^{(i)}_{u_i}),
\end{eqnarray}

\begin{eqnarray}\label{LBr}
 L(N,M) = \max_{\textbf{u} \in \Omega^c_{N,M}} {\bf Br}(\textbf{u}),
\end{eqnarray}

\begin{eqnarray}\label{ZBr}
 Z^{{\bf Br}}_{\beta}(N,M) = P^c_{N,M} \left( e^{\beta {\bf Br}(\textbf{u})} \right). 
\end{eqnarray}

\noindent  The functional (\ref{LBr}) is the aforementioned Brownian percolation problem from queueing theory.
Observe that it has the interesting property that

\begin{eqnarray*}
  L(N,M) = \sqrt{N }L(1,M),
\end{eqnarray*}

\noindent in law. This is due to the scaling properties of Brownian motions. It is now a well known fact
that $L(1,M)$ has the same law as the larger eigenvalue of a Gaussian Unitary random matrix (GUE, see \cite{Bar, OY} among other proofs).
As a consequence,

\begin{eqnarray}\label{TW}
 N^{1/6} \left( L(1,N) - 2N^{1/2} \right) \longrightarrow F_{2},
\end{eqnarray}

\noindent where $F_{2}$ denotes the Tracy-Widom distribution \cite{TW1}. It describes the fluctuations of the top
eigenvalue of the GUE and its distribution function can be expressed as

\begin{eqnarray*}
 F_2(s) = \exp \left\lbrace - \int^{+\infty}_s (x-s)u(s)^2 dx \right\rbrace,
\end{eqnarray*} 

\noindent where $u$ is the unique solution of the Painlevé II equation

\begin{eqnarray*}
 u'' = 2u^3 + xu,
\end{eqnarray*}

\noindent with asymptotics 

$$u(x) \sim \frac{1}{2\sqrt{\pi} x^{1/4}} \exp \left\lbrace -\frac{2}{3} x^{3/2} \right\rbrace.$$

\noindent The distribution function $F_2$ is non-centered and its asymptotics behavior is
as follows:

\begin{eqnarray*}
  F_2(s) \sim e^{\frac{1}{12} s^3},\, {\rm as} \, s\to -\infty, \quad 1-F_2(s) \sim e^{-\frac{4}{3} t^{3/2}}\, {\rm as}
  \, t\to +\infty.
\end{eqnarray*}

\noindent See \cite{AGZ} for more details about the Tracy-Widom distribution and random matrices in general.
In the discrete setting, it is shown in \cite{BM} that, for $M=N^a$ with $0<a<3/7$,

\begin{eqnarray*}
 \frac{T(N,N^a)-2N^{(1+a)/2}}{N^{1/2 - a/6}} \longrightarrow F_{2}.
\end{eqnarray*}

\noindent The proof uses similar approximations than the seminal work of Glynn and Whitt. See also \cite{BSu} for
similar results.

\vspace{2ex}

 The third display (\ref{ZBr}) is the partition function of the continuous-time directed polymers in Brownian environment. The 
free energy of this polymer model is explicit. Its exact value was first conjectured in \cite{OY} based on a generalized version of the Burke's 
Theorem and detailed heuristics. The proof was then completed in \cite{MOC}: 

\begin{theorem}[Moriarty-O'Connell]\label{MO-th}\cite{MOC}
\begin{eqnarray}
 \lim_{N\to +\infty} \frac{1}{N} \log Z^{{\bf Br}}_{\beta}(N,N) = f(\beta),
\end{eqnarray}

\noindent where
\begin{eqnarray}
	f(\beta)=
	\left\{ \begin{array}{ll}-(-\Psi)^*(-\beta^2)-2\log |\beta| & :\beta \neq 0\\ 0 & : \beta=0	
		\end{array} \right.
	\label{putz}
\end{eqnarray}

\noindent where $\Psi(m) \equiv \Gamma'(m)/\Gamma(m)$ is the restriction of the digamma function to $(0,+\infty)$, $\Gamma$
is the Gamma function

$$ \Gamma(m) = \int^{+\infty}_0 t^{m-1} e^{-t} dt,$$

\noindent and $(-\Psi)^*$ is the convex dual of the function $-\Psi$:

$$ (-\Psi)^*(u) = \inf_{m\geq 0} \left\lbrace mu + \Psi(m) \right\rbrace.$$
\end{theorem}

\vspace{2ex}

We now search for a 'regime' in which the limiting free energy of the discrete model is the same as the Brownian one. It turns out that a way to achieve this is to increase the temperature in the asymmetric discrete model, as $N$ tends to $+\infty$.
So the Moriarty-O'Connell polymer can be viewed as an approximation of a discrete polymer close to an axis at a very high temperature.

\begin{theorem}[The Moriarty-O'Connell regime]\label{MOregime}
 Let $\beta_{N,a} = \beta N^{(a-1)/2}$,

\begin{eqnarray}\label{ZMO}
 \lim_{N\to +\infty} \frac{1}{\beta_{N,a}N^{(1+a)/2}} \log Z_{\beta_{N,a}}(N,N^a) = f(\beta)/\beta
\end{eqnarray}

\end{theorem}

In Section \ref{asymmetric}, we will give a proof of a $d$ dimensional version of this fact. Unfortunately, we are no longer 
able to compute explicitly the free energy for the Moriarty-O'Connell model when $d\geq 2$. We can even treat more asymmetric cases, where the additional asymmetry
translates in a lost of dimensions in the limit (Section \ref{vac}). The proof of this fact is closely related to the continuity of the free energy of point-to-point directed polymers
at fixed temperature at the border of an octant wich is discussed in Section \ref{cptp}.

\vspace{3ex}
We then turn to the problem of computing the free energy of a directed polymers model with a drift that grows with $N$.
Let

\begin{eqnarray}\label{Zdrift}
  Z^{(h)}_{\beta,N} = \sum_{1\leq n \leq N} \oZ_{\beta}(n,N-n) e^{-h \times (N-n)},
\end{eqnarray}

\noindent where, for each $n$, $\oZ_N(n,N-n)$ is the (non-normalized) point-to-point
partition function

\begin{eqnarray}\label{nnptp}
 \oZ_{\beta}(n,N-n) = \sum_{\omega \in \Omega_{n,N-n}} e^{\beta H_N(\omega)}.
\end{eqnarray}

\noindent This can also be seen as a generating function or a Poissonization of the point-to-point partition function.
Recall that, when $N-n = O(N^a)$, Theorem \ref{2LPP} implies that

\begin{eqnarray*}
 \lim_{N \to +\infty} \frac{1}{\sqrt{n(N-n)}} \log \oZ_{\beta}(n,N-n) = 2\beta,
\end{eqnarray*}

\noindent as, in this regime, $\log |\Omega_{n,N-n}|$ is of much smaller order than $\sqrt{N(N-n)}$ (see also (\ref{thermo-thbox})).
The role of the drift $h$ in (\ref{Zdrift}) is to penalize the paths for which the final point is far from the horizontal axis. 
It has to be calibrated in order to favor final points such that $N-n=O(N^a)$.

\vspace{2ex}

Our first result for this model concerns the value of the free energy.

\begin{theorem}\label{strongdrift}
Take $h=h_N= \gamma N^{(1-a)/2}$. Then,
 \begin{eqnarray*} 
  \lim_{N\to +\infty} \frac{1}{N^{(1+a)/2}} \log Z^{(h_N)}_{\beta,N} = \frac{\beta^2}{\gamma},
 \end{eqnarray*}

\noindent for all environment laws such that $Q(e^{\beta \eta})< +\infty$ for all $\beta>0$.

\end{theorem}

We can even give the correct order of the fluctuations of the free energy. The bounds we obtain have a certain flavor of variance bounds without being exactly such. 

\begin{theorem}\label{fluct23}
For all $a<1/3$ ($a<3/7$ for a Gaussian environment), there exists a constant $C>0$ such that, for all $N\geq 1$,

 \begin{eqnarray*}
  \frac{1}{C} N^{1-a/3} \leq 
Q \left\lbrace \left( \log Z^{h_N}_{\beta, N}-\frac{\beta^2}{\gamma} N^{(1+a)/2} \right)^2 \right\rbrace \leq C N^{1-a/3}.
 \end{eqnarray*}
\end{theorem}

\noindent The proof is based on non-asymptotic deviation inequalities for the partition function. These are reminiscent of similar bounds for random matrices proved by Ledoux and Rider (\cite{LR}). Similar bounds were obtained in the context of LPP in \cite{I} for Gaussian or Bounded environments. We strongly believe that the properly rescaled fluctuations should converge to the Tracy-Widom distribution. However, we would need a more precise analysis to prove this afirmation (see Remark \ref{remark-TW}).
Note that the recent article \cite{Sepp-pol} includes fluctuation bounds for a (symmetric) one-dimensional model of directed polymers in a log-Gamma environment.

\vspace{1ex}


\vspace{3ex}

The rest of this work is organized as follows: 

\begin{itemize}

\item We prove the continuity of the point-to-point partition function for discrete models in Section \ref{cptp}. 

\item In Section \ref{dpbe}, 
we discuss the existence of the free energy
for the directed polymers in Brownian environments. 

\item In Section \ref{asymmetric}, we discuss the links between asymmetric directed polymers and
directed polymers in a Brownian environment. We give the proof of a multidimensional version
of Theorem \ref{MOregime} in Section \ref{asymmetric} and discuss a more asymmetric
situation in Section \ref{vac}.

\item Finally, we study the model of directed polymers with
a huge drift in Section \ref{phd}. Theorem \ref{strongdrift} is proved in Subsection \ref{sdfe}, while
the fluctuation bounds (Theorem \ref{fluct23}) are proved in Subsection \ref{sdfb}.
\end{itemize}

\vspace{2ex}

\noindent An extended version of this article can be found in the Thesis \cite{Mthesis}. It includes a complete review of the litterature about Brownian 
percolation and one-dimensional directed polymers in a Brownian environment. 

\subsection*{Aknowledgments} Most of this work was done while I was a PhD student at Paris 7. I would like to thank my advisor Francis Comets for his kind guidance. I also would like to thank Jean-Paul Ibrahim, Neil O'Connell and Nikolaos Zygouras for many valuable discussions and their comments on an early draft of this work, and to Thierry Bodineau for pointing out the reference \cite{Martin}.



\section{Continuity of the point-to-point partition function for the discrete model}\label{cptp}
We prove here the continuity of the point-to-point free energy seen as a function
from the octant $\{\nx \in \re^d:\, x_i \geq 0\}$ to $\re$. Only the
continuity at the boundary of the octant requires a proof, as the continuity in the interior is an easy consequence
of the concavity properties of the free energy (which itself follows from sub-additivity). 

For $\ny \in \re^d_+$, define 

\begin{eqnarray*}
 Z^{\beta}_N(\ny) = \sum_{\ns \in \Omega_{N \ny}} \exp \beta H(\ns),
\end{eqnarray*}

\noindent where $\Omega_{N\ny}$ is the set of directed paths from the origin to $N \ny$, which, by notational abuse, denotes
the point in $\ze^d$ which $i$-th coordinate is $\lfloor N y_i \rfloor$. Note that the dimension here is $d$ and not $d+1$ as usual.
We will be interested in directions
of the form $\ny_h = (h,\nx)$ with $\nx \in \re^{d-1}_+$ (i.e. $\nx \in \re^{d-1}$, $x_i >0$), and $h\geq 0$.
In this case, we just denote the partition function by $Z_N(h,\nx)$.
We also define the point-to-point free energy:

\begin{eqnarray*}
 \psi^{\beta}(\ny) = \lim_{N\to +\infty} \frac{1}{N} \log Z^{\beta}_N(\ny),
\end{eqnarray*}

\noindent and we adopt the convenient notation $\psi(h,\nx)$ for $\psi(\ny_h)$ (we also dropped the dependence in $\beta$).
$\psi$ is a function from the octant $\{\nx \in \re^d:\, x_i \geq 0\}$ to $\re$.

\begin{proposition}\label{prop-cont-PTP}
  \begin{eqnarray*}
    \lim_{h\downarrow 0} \psi(h,\nx) = \psi(0,\nx).
  \end{eqnarray*}
\end{proposition} 

\begin{proof}
Each path from the origin to $N (h,\nx)$ can be decomposed into $Nh$ segments with constant first coordinate: for each path, 
there is a collection of points $(\nm_i)_{i\leq Nh}$ with $\nm_i \in \ze^{d-1}_+$ and such that
for each $0\leq i < Nh$, there is a segment of the path linking $(i,\nm_i)$ and $(i,\nm_{i+1})$.
So the partition function can be decomposed itself as

\begin{eqnarray}\label{subad-PTP}
 Z_N(h,\nx) = \sum_{(\nm_i)_i} \Pi_i Z(i; \nm_i,\nm_{i+1}),
\end{eqnarray}

\noindent where, for each $i$, $Z(i; \nm_i,\nm_{i+1})$ is a sum over directed paths linking $(i,\nm_i)$ and $(i,\nm_{i+1})$.
The collection of possible points $(\nm_i)_i$ runs over a set $J^{N}_{h,\nx}$ which cardinality satisfies
$\log |J^{N}_{h,\nx}| = N\phi(h,\nx) + o(N)$ for some $\phi(h,\nx) \to 0$ as $h\to 0$ (see remark below). 
We will analyze each summand of the right hand side of (\ref{subad-PTP}) separately:

\begin{eqnarray}\nonumber
 &&Q \left( \log \Pi_i Z(i; \nm_i,\nm_{i+1}) \right) =  \sum_i Q\left(\log  Z(i; \nm_i,\nm_{i+1})\right)\\
 \nonumber
 &\ & \quad \quad =  \sum_i Q\left(\log  Z(0; \nm_i,\nm_{i+1})\right)\\
 \nonumber
&\ & \quad \quad \leq Q \left( \log Z_N(0,\nx) + \beta \sum \eta(i,\nm_{i+1}) \right)\\
&\ & \quad \quad \leq  N \, \psi(0,\nx). \label{boundp-PTP}
\end{eqnarray}

\noindent The second equality follows by translation invariance; in the third line, we use the fact that the partition 
functions do not consider the environment at the starting point; the last inequality follows by subadditivity, as

\begin{eqnarray*}
 \psi(\ny) = \sup_N \frac{1}{N} Q \log Z_N(\ny),
\end{eqnarray*} 

\noindent and the fact that $Q\eta = 0$. Now, the concentration inequality 
implies that

\begin{eqnarray}\nonumber
  Q\left( \left| \log \Pi_i Z(i; \nm_i,\nm_{i+1}) - Q\log \Pi_i Z(i; \nm_i,\nm_{i+1})\right| \geq \epsilon N \right)
 \leq e^{-c \epsilon^2 N}.\\
 \label{conc-PTP}
\end{eqnarray}

\noindent for $\epsilon$ small enough (see \cite{LW} and \cite{CY_branch} Proposition 3.2.1-b). Using (\ref{subad-PTP}), we can see that, if

\begin{eqnarray*}
  \log Z_N(h,\nx) \geq N \, \psi(0,\nx) + \epsilon N,
\end{eqnarray*}

\noindent for some $\epsilon > 0$, then, for some $(\nm_i)_i \in J^N_{h,\nx}$, it must happen that

\begin{eqnarray*}
  \log \Pi_i Z(i; \nm_i,\nm_{i+1}) \geq N \, \psi(0,\nx) + \epsilon N - \log |J^N_{h,\nx}|.
\end{eqnarray*}

\noindent By (\ref{boundp-PTP}), this means that the quantity in the left hand side deviates more than
$ \epsilon N - \log |J^N_{h,\nx}|$ from its mean. By the asymptotics on $|J^N_{h,N}|$, for $h$ small enough, we will
have that $\log |J^N_{h,\nx}| < \epsilon N/2$, and then the inequality (\ref{conc-PTP}) applies. Then,

\begin{eqnarray*}
 Q\left(  \log Z_N(h,\nx) \geq N \, p(0,\nx) + \epsilon N \right) 
 \leq \exp \left\lbrace N\phi(h,\nx) - c\epsilon^2 N + o(N) \right\rbrace. 
\end{eqnarray*}

\noindent By taking $h$ even smaller if necessary, the right hand side of this inequality becomes summable. By
Borel-Cantelli we will then have that

\begin{eqnarray*}
   \log Z_N(h,\nx) \leq N \, \psi(0,\nx) + \epsilon N,
\end{eqnarray*} 

\noindent $Q$-almost surely for $N$ large enough. Dividing both sides by $N$ and taking the limit $N\to +\infty$,
we conclude that

\begin{eqnarray*}
 \psi(h,\nx) \leq \psi(0,\nx) + \epsilon,
\end{eqnarray*}

\noindent for $h$ small enough. We now have to check the reverse inequality. But it follows easily that

\begin{eqnarray*}
 \log Z_N(h,\nx) \geq  \log Z_N(0,\nx) + \beta \sum^{hN}_{i=0} \eta(N\nx, i).
\end{eqnarray*}

\noindent Recalling that the $\eta$'s are centered, dividing by $N$ and taking the limit
$N\to +\infty$ give that $\psi(h,\nx) \geq \psi(0,\nx)$.
\end{proof}

\begin{remark}\label{rem-phi}
 The function $\phi$ can be made explicit: as 

 \begin{eqnarray*}
  \log |J^{N}_{h,\nx}| = \prod^d_{i=2} \left( {}^{\lfloor Nx_i \rfloor + \lfloor Nh \rfloor}_{\quad \lfloor Nh \rfloor}\right),
 \end{eqnarray*}

\noindent by Stirling formula, we have $ \log |J^{N}_{h,\nx}| = N \phi(h,\nx) + o(N)$, with

 \begin{eqnarray*}
  \phi(h,\nx) = \sum_{{}^{2\leq i \leq d}_{x_i > 0}} \left( h \log \frac{x_i + h}{h} + x_i \log \frac{x_i + h}{x_i} \right).
 \end{eqnarray*}
\end{remark}

\begin{cor}
  The point-to-point free energy is continuous on $\re^d_+$.
\end{cor}
\begin{proof}
  The continuity in the interior of $\ze^d_+$ is a consequence of the concavity properties arising from the subadditivity.
See the proof of Theorem \ref{free-energy-d} where this is explained in the continuous setting.
The continuity at the boundary follows from repeated use of the preceding Proposition.
\end{proof}

\begin{remark}
In the one-dimensional case, a very precise asymptotic for the last-passage percolation is available. It implies that
$\psi(1,h) = 2\sqrt{h} + o(\sqrt{h})$ as $h \downarrow 0$ (see \cite{Martin}, Theorem $2.3$). 
\end{remark}

\begin{remark}
 This scheme of proof will reappear later in the proof of a certain continuity at the borders property for very asymmetric directed
polymers, in the regime where the limit is the Brownian free energy.
\end{remark}


\section{Directed Polymers in a Brownian Environment}\label{dpbe}

We will now generalize the Brownian setting introduced before to larger dimensions.

Let $\nx\in \ze^ d$ such $x_i \geq 1$ for all $i=1,...,d$. Let $M=\sum^d_{i=1} x_i$. This is
basically the length of a nearest-neighbor path from the origin ${\bf 0}$ to $\nx$.
Let $\Omega^c_{t,\nx}$ be the set of right-continuous paths $\ns$ such that:

$(i)$ $\ns_0={\bf 0}$ and $\ns_t = \nx$.

$(ii)$ $\ns$ performs exactly $M$ jumps, according to the coordinate vectors. 

\noindent So the {\it skeleton} of $\ns$
can be thought of as a discrete nearest-neighbor path form the origin to $\nx$. $\ns$ itself can be viewed as a directed
path in $\re^+ \times \ze^{d}$ starting from the origin at time $0$ and reaching the site $(t,{\nx})$ at time $t$. Let $P^c_{t,\nx}$ be the uniform measure on $\Omega^c_{t,\nx}$.

Now consider a family $\{ B(\ny):\, \ny\in \Lambda_{\nx}\}$ of independent Brownian motions, where $\Lambda_{\nx} = \{ \ny\in \ze^d: 0\leq y_i \leq x_i,\,
\forall \, i=1,...,d\}$. Define the energy of a path ${\bf s}$ in the following way: let $0=t_0 < t_1 < \cdots < t_M < t$ be the
jumps times of ${\bf s}$ and put $t_{M+1}=t$, then

\begin{eqnarray}\label{Br-d}
\nBr(\ns) = {\bf Br}(t,\nx)({\bf s}) = \sum^{M+1}_{k=1} \left( B_{t_k}({\bf s}_{t_k}) - B_{t_{k-1}}({\bf s}_{t_k}) \right).
\end{eqnarray}

\noindent The partition function of the directed polymers in Brownian environment at inverse temperature $\beta$ is

\begin{eqnarray}
Z^{\nBr}_{\beta}(t,\nx) = P^c_{t,\nx} \left( \exp \beta \nBr(\ns) \right).
\end{eqnarray}

\noindent We first prove the existence of the free energy in the linear regime.
 Take $\alpha \in \re^d$ with strictly positive entries.

\begin{theorem}\label{MO-th-d}
Let ${ \alpha} N$ be the point of $\ze^d$ whose $i$-th coordinate is equal to $\lfloor \alpha_ i N \rfloor$. Then the following 
deterministic limit

\begin{eqnarray}\label{free-energy-d}
p(\beta, \alpha, d) =\lim_{N \to +\infty} \frac{1}{N} \log Z^{\nBr}_{\beta}(N, {\alpha N} )
\end{eqnarray}
\noindent exists $Q$-a.s.. Moreover, the function $\alpha \mapsto p(\beta, \alpha, d)$ is continuous on its domain.

\end{theorem}

\begin{proof}
 First, fix $\alpha$. The proof uses subadditivity. To lighten notation, denote $|\Omega_N|$ for $|\Omega_{N,{\alpha N}}|$. 
 We consider unnormalized versions of the partition function:
 
 \begin{eqnarray*}
  \int_{\Omega_{N+M}} e^{\beta \nBr(N+M, \nx_{N+M})(\ns)} 
  &\geq & \int_{\Omega_{N+M}} e^{\beta \nBr(N+M, \nx_{N+M})(\ns)} {\bf 1}_{\ns_N=\alpha N}\\
  &=& \int_{\Omega_N} e^{\beta \nBr(N, { \alpha N})(\ns)} \times
      \left( \int_{\Omega_M}e^{\beta \nBr(M, \alpha M)(\ns)}\right) \circ \theta_{N,\alpha N},
 \end{eqnarray*}
 
 \noindent where the shift $\theta_{k,\nx}$ means that we use the Brownian motions
 
 \begin{eqnarray*}
  \overline{B}^{(\ny)}(\cdot)=B^{(\ny + \nx)}(\cdot + k),
  \end{eqnarray*}
  
 \noindent to define $\nBr$. By subadditivity, it follows that there exists a deterministic function
 $\overline{p}(\beta, \alpha, d)$ such that
 
 \begin{eqnarray*}
  \overline{p}(\beta, \alpha, d) = \lim_{N\to +\infty} \frac{1}{N} \log \int_{\Omega_N} e^{\beta \nBr(N, \alpha N)(\ns)},
 \end{eqnarray*}
 
 \noindent $Q$-almost surely. Apply this with $\beta=0$ and the theorem follows with 
 $p(\beta, \alpha, d)= \overline{p}(\beta, \alpha, d) / \overline{p}(0, \alpha, d)$. Now, take $\alpha_1$ and $\alpha_2$ in $\re^d$ with strictly positive coordinates, and $\lambda \in (0,1)$.  Then,

\begin{eqnarray*}
Z^{\nBr}_{\beta}\left(N, N(\lambda \alpha_1 + (1-\lambda) \alpha_2)\right) \geq 
Z^{\nBr}_{\beta}\left(\lambda N,\lambda \alpha_1 N\right) \times Z^{\nBr}_{\beta}\left((1-\lambda)N, (1-\lambda) \alpha_2 N\right) \circ \theta_{\lambda N, \lambda \alpha_1 N}.
\end{eqnarray*}

\noindent Taking logarithms in both sides, dividing by $N$ and taking limits, leads to,

\begin{eqnarray*}
\overline{p}\left(\beta,\lambda \alpha_1 + (1-\lambda) \alpha_2, d\right) 
\geq \lambda \overline{p}\left(\beta, \alpha_1, d \right) 
+ (1-\lambda) \overline{p}\left(\beta, \alpha_2, d\right) 
\end{eqnarray*}

\noindent So $\alpha \mapsto \overline{p}(\beta, \alpha, d)$ is concave, and then continuous. As $p(\beta, \alpha, d) = \overline{p}(\beta, \alpha, d)/
\overline{p}(0, \alpha, d)$, it is also continuous.
\end{proof}

\begin{rem}
 Note that as we have true subadditivity, we can avoid the use of concentration. However, we can state the following result:

\begin{eqnarray}\nonumber
 Q \left( \left| 
\log Z^{\nBr}_{\beta}(N, \alpha N) - Q \log Z^{\nBr}_{\beta}(N, \alpha N) \right|> uN \right)
\leq C \exp \left\lbrace - \frac{Nu^2}{C \beta^2} \right\rbrace.\\
\label{conc-d}
\end{eqnarray}

This can be proved as Formula $(9)$ in \cite{RT}, using ideas from Malliavin Calculus.
\end{rem}


\section{Asymmetric Directed Polymers in a Random Environment}\label{asymmetric}

The central part of this Section is the proof of a multidimensional version of Theorem
\ref{MOregime}. 


Let $\nx\in \ze^ d$ such $x_i \geq 1$ for all $i=1,...,d$ and $N\geq 1$. Let $M=\sum^d_{i=1} x_i$ be the
distance between the origin and $x$ in $\ze^ d$. Let $\Omega_{N,\nx}$ be the
set of directed paths from the origin in $\ze^{d+1}$ to $(N,x)$ that is

\begin{eqnarray*}
\Omega_{N,\nx} = \{ \nS:\{0,...,N+M\} \to \ze^{d+1}: \nS_0=0,\, \nS_{N+M}=(N,\nx),\\ \forall \, t,\, 
\nS_{t+1}-\nS_t \in \{ e_i:\, i=1,...,d \} \}.
\end{eqnarray*}
 
 \noindent Consider a collection of i.i.d. random variables $\{ \eta(k,\nx):\, k\in \ze,\, \nx\in \ze^d\}$.
 We will assume that $Q(e^{\beta \eta})<+\infty$ for all $\beta \geq 0$.
 For a fixed realization of the environment, define the energy of a path $\nS\in \Omega_{N,\nx}$ as
 
 \begin{eqnarray}\label{Ham-d}
 H(\nS) = \sum^{N+M}_{t=1} \eta(\nS_t).
 \end{eqnarray}
 
 \noindent The polymer measure at inverse temperature $\beta$ is now defined as the measure on $\Omega_{N,\nx}$ such that
 
 \begin{eqnarray*}
 \frac{d\mu_{N,\nx}}{dP_{N,\nx}}(\nS) = \frac{1}{Z_{\beta}(N,\nx)} \exp \beta H(\nS),
 \end{eqnarray*}
 
\noindent where $Z_{\beta}(N,\nx)$ is the point-to-point partition function

\begin{eqnarray*}
Z_{\beta}(N,\nx) = P_{N,\nx} \left( \exp \beta H(\nS) \right).
\end{eqnarray*}

\vspace{2ex}

We will be interested in the limit as $N$ grows to infinity and $\nx=\nx_N$, with $|\nx_N| \to +\infty$ with $N$ in an
appropriate way. Take $\alpha \in \re^d$ with strictly positive coordinates. Let $\alpha N^a$ be the point in $\ze^d$ 
which $i$-th coordinate is equal to $\lfloor \alpha_i N^a \rfloor$. The following theorem
is the generalization to $\ze^d$ of Theorem \ref{MOregime}.

\begin{theorem}\label{MOregime-d}
 Let $\beta_{N,a} = \beta N^{(a-1)/2}$. Then,

\begin{eqnarray}\label{ZMO-d}
 \lim_{N\to +\infty} \frac{1}{\beta_{N,a}N^{(1+a)/2}} \log Z_{\beta_{N,a}}(N,\alpha N^a) = p(\beta, \alpha, d)/\beta,
\end{eqnarray}

\noindent$Q$-almost surely, where $p(\beta, \alpha, d)$ is the free energy of the continuous-time directed polymer in a Brownian environment
as in (\ref{free-energy-d}).

\end{theorem}

\vspace{2ex}

\begin{remark}
 To lighten notation, in the following, $C$ will denote a generic constant whose value can vary from line to line.
 Also, we can consider $\alpha = (1,\cdots,1)$ for simplicity and introduce the notations $\Omega_{N,a} = \Omega_{N,N^a}$ and $P_{N,a}=P_{N,N^a}$, and 
similarly for their continuous counterparts.
\end{remark}

\subsection{Proof of Theorem \ref{MOregime-d}}
\vspace{2ex}

The proof is carried on in $4$ Steps.
Much of the computations in Steps $1$ and $2$ are inspired by \cite{BM,I}, while the scaling argument
in Step $3$ is already present in \cite{GW}.

\subsubsection{First Step: approximation by a Gaussian environment}
The central ingredient of this part of the proof is a strong approximation technique
by Koml\' os, Major and Tusn\' ady: 
 let $\{ \eta_t: t\geq 0\}$ denote an i.i.d. family of random variables, with $Q(\eta_0)=0$, $Q(\eta^2_0)=1$
and $Q(e^{\beta \eta_0})< +\infty$ for all $0\leq \beta \leq \beta_0$ for some $\beta_0 > 0$. Let $\{ g_t: t\geq 0\}$ denote an i.i.d. family of standard normal variables. Denote 

\begin{eqnarray*}
 S_N = \sum^N_{t=0} \eta_t, \quad T_N = \sum^N_{t=0} g_t.
\end{eqnarray*}

\begin{theorem}[KMT approximation]\label{KMT-th} \cite{KMT2}
The sequences $\{\eta_t:t\geq 0\}$ and $\{g_t:t\geq 0\}$ can be constructed in such a way that, for all $x>0$ and every $N$,
\begin{eqnarray}\label{KMT}
 Q \left\lbrace \max_{k\leq N} |S_k - T_k| > K_1 \log N + x \right\rbrace \leq K_2 e^{- K_3 x},
\end{eqnarray}

\noindent where $K_1, \, K_2$ and $K_3$ depend only on the distribution of $\eta$, and $K_3$ can be taken as large as
 desired by choosing $K_1$ large enough. Consequently, $|S_N-T_N|= O(\log N)$, $Q$-a.s..
\end{theorem}

\vspace{2ex}
Now consider our environment variables $\{ \eta(t,\nx): t\in \ze,\nx \in \ze^d \}$. Use Theorem \ref{KMT-th} to couple each
 'row' $\eta(\cdot,\nx)$ with standard normal variables $g(\cdot,x)$ such that

\begin{eqnarray*}
 Q \left\lbrace \max_{k\leq N} |S(k,\nx) - T(k,\nx)| > C\log N + \theta \right\rbrace \leq K_2 e^{- K_2 \theta},\quad \forall \theta> 0,
\end{eqnarray*}

\noindent where $ S(k,\nx) = \sum^k_{t=0} \eta(t,\nx)$ and  $T(k,\nx) = \sum^k_{t=0} g(t,\nx)$.

\vspace{2ex}

Now, we need to decompose each path $\nS \in \Omega_{N,a}$ into its 'jump' times $\nT=(T_i)_i$ and its position between jump times $\nL=(L_i)
_i$. We say that $T$ is a jump time if one of the coordinates of $\nS$ other than the first changes between instants $T-1$
and $T$. We can order the jump times of $\nS$: $T_0=0< \cdots < T_{dN^a} < T_{dN^a+1}=N$. We can then define $L_i$ as the
point $\ny \in \ze^d$ such that $\nS_{T_i} = (T_i, \ny)$. We can rewrite the Hamiltonian (\ref{Ham-d}) as

\begin{eqnarray*}
H(\nS) = \sum^{dN^a}_{i=0} \Delta H(\nS,i),
\end{eqnarray*}

\noindent where

\begin{eqnarray*}
\Delta H(\nS,i) = \sum^{T_{i+1}-1}_{k=T_i} \eta(k,L_i).
\end{eqnarray*}

\noindent Define $g(\nS)$ and $\Delta g(\nS,i)$ just in the same way by replacing the variables $\eta$ by the Gaussians $g$.
Then,

\begin{eqnarray*}
 |H(\nS)-g(\nS)| \leq \sum^{dN^a}_{i=0} |\Delta H(\nS,i)-\Delta g(\nS,i)|.
\end{eqnarray*}

\noindent Let $\theta_N$ be an increasing function to be determined later and
$\Lambda_{N,a} = \{\ny \in \ze^d:\, 0\leq y_i \leq \lfloor  N^a \rfloor\}$:

\begin{eqnarray*}
 &\ & Q \left\lbrace |H(\nS)-g(\nS)| > 2dN^a \theta_N, \, {\rm for}\, {\rm some}
\, \nS \in \Omega_{N,a} \right\rbrace \\
&\leq& Q \left\lbrace \sum^{dN^a}_{i=1} |\Delta H(\nS,i)-\Delta g(\nS,i)| > 2dN^a \theta_N, \, {\rm for}\, {\rm some} \, 
 \nS \in \Omega_{N,a} \right\rbrace 
\end{eqnarray*}

\begin{eqnarray*}
&\leq & Q \left\lbrace \max_{i\leq dN^a}|\Delta H(\nS,i)-\Delta g(\nS,i)| > 2\theta_N, \, {\rm for}\, {\rm some}
\, \nS \in \Omega_{N,a} \right\rbrace \\
&\leq & Q \left\lbrace \max_{k\leq N} |S(k,\nx)- g(k,\nx)| ,\, {\rm for}\, {\rm some}\,
\nx \in \Lambda_{N,a} \right\rbrace\\
&\leq& |\Lambda_{N,a}| Q \left\lbrace \max_{k\leq N} |S_k - T_k| > \theta_N \right\rbrace,
\end{eqnarray*}

\noindent In order to apply Theorem \ref{KMT-th}, we have to take $\theta_N = K_1\log N + \epsilon_N$, and to apply Borel-Cantelli,
as $|\Lambda_{N,a}|\leq  N^{da}$, it is enough to take $\epsilon_N = c \log N$ with $c$ large enough to make $N^a e^{-K_3 \epsilon_N}$ summable.
Then, $Q$-a.s., $|H(\nS) - g(\nS)| \leq C N^a \log N $ for all $\nS \in \Omega_{N,a}$, for $N$ large enough. This shows that

\begin{eqnarray*}
 P_{N,a} \left( e^{\beta_{N,a} H(\nS)} \right) = P_{N,a} \left( e^{\beta_{N,a} g(\nS)} \right) O(e^{C\beta_{N,a} N^a \log N}).
\end{eqnarray*}

\vspace{2ex}
\noindent Recall that $\beta_{N,a} = \beta N^{(a-1)/2}$, so that $\beta_{N,a} N^a \log N = O(N^{(3a -1 )/2} \log N)$.
As $0<a<1$, we have $\beta_{N,a} N^a \log N << \beta_{N,a} N^{(a+1)/2} = \beta N^a$, and then

\begin{eqnarray*}\label{Z-KMT}
 \log Z_{\beta_{N,a}}(N,\alpha N^a) = \log Z^g_{\beta_{N,a}}(N,\alpha N^a) + o(\beta_{N,a} N^{(1+a)/2}),
\end{eqnarray*}

\noindent where the superscript $g$ means that the environment is Gaussian.

\noindent We can conclude that, if the limit free energy exists $Q$-a.s. for Gaussian environment variables, it exists for all environment variables having some finite exponential moments, and the limit is the same.

\vspace{2ex}
\subsubsection{Second Step: approximation by continuous-time polymers in Brownian environment}
Having replaced our original disorder variables by Gaussians, we can take them
as unitary increments of independent one-dimensional Brownian motions.
We then just have to control their fluctuations to replace the discrete paths by continuous paths in a Brownian environment. This is what will be done in the following paragraphs.

We first need to establish a correspondence between continuous paths an discrete ones.

Take ${\bf s}\in \Omega^c_{N,a}$, and recall the definition (\ref{Br-d}) for the Brownian Hamiltonian ${\bf Br}(\ns)$
and that $0=t_0 < t_1 < \cdots < t_{dN^a + 1}=N$ denote the jump times of
$\ns$. Let $l_i = \ns_{t_i}$.
The path ${\bf s}$ can be discretized by defining the following Gaussian Hamiltonian:

\begin{eqnarray}\label{Br-discrete}
 H^{{g}}(\ns) = 
\sum^{dN^a}_{i=0} \left( B^{(l_i)}_{\lfloor t_{i+1}\rfloor} - B^{(l_i)}_{\lfloor t_{i} \rfloor - 1} \right). 
\end{eqnarray}

\noindent This is equivalent to consider $g(\nS)$ where $\nS \in \Omega_{N,a}$ is defined through its jump times $T_i$ and successive positions $L_i$ by

\begin{eqnarray*}
T_i &=& \lfloor t_i \rfloor ,\\
L_k &=& l_i,\quad \forall \, T_i\leq k < T_{i+1},
\end{eqnarray*}

\noindent (Recall that the Gaussian variables obtained in the previous step are now {\it embedded} in the Brownian motions).
In this way,

\begin{eqnarray*}
 P^c_{N,a}\left(\exp {\beta H^{{\bf Br}}(\ns)}\right) = P_{N,a}\left(\exp {\beta g(\nS)}\right).
\end{eqnarray*}

\noindent We have now to approximate the previous expression by $Z^{{\bf Br}}_{\beta}(N,N^a)$. 
Take $\ns\in \Omega^c_{N,a}$:

\begin{eqnarray*}
 |H^{g}(\ns) &-& {\bf Br}(\ns)| =
\left|\sum^{dN^a}_{i=0} \left(B^{(l_i)}_{\lfloor t_{r+1}\rfloor } - B^{(l_i)}_{\lfloor t_{r} \rfloor -1}\right) 
-\sum^{dN^a}_{i=0} \left( B^{(l_i)}_{ t_{i+1} } - B^{(l_i)}_{ t_{i}} \right) \right|\\
&\leq& \sum^{dN^a}_{i=0} \left|B^{(l_i)}_{\lfloor t_{r+1}\rfloor } - B^{(r)}_{ t_{r+1} }\right|+
\sum^{dN^a}_{i=0} \left|B^{(l_i)}_{\lfloor t_{i}\rfloor } - B^{(l_i)}_{ t_{i}-1 }\right|\\
&\leq& 2 \sum^{dN^a}_{i=0} \sup_{{}^{0\leq s,\, t \leq N+1}_{\quad |s-t|<2}} |B^{(l_i)}_s - B^{(l_i)}_t|.
\end{eqnarray*}

\noindent This can be handled with basic properties of Brownian motion: denote by $x_N$ an increasing function to be determined,

\begin{eqnarray*}
 &\ &Q \left( \sum^{dN^a}_{i=0} \sup_{{}^{0\leq s,\, t \leq N+1}_{\quad |s-t|<2}} |B^{(l_i)}_s - B^{(l_i)}_t| > dN^a x_N,\,
 {\rm for}\, {\rm some}\, \ns \in \Omega_{N,a} \right)\\
&\leq& Q\left( \max_{1\leq i \leq dN^a} \sup_{{}^{0\leq s,\, t \leq N+1}_{|s-t|<2}} |B^{(l_i)}_s - B^{(l_i)}_t|> x_N\,
{\rm for}\, {\rm some}\, \ns \in \Omega_{N,a} \right)\\
&\leq& Q\left( \max_{\nx \in \Lambda_{N,a}} \sup_{{}^{0\leq s,\, t \leq N+1}_{|s-t|<2}} |B^{(\nx)}_s - B^{(\nx)}_t|> x_N \right)\\
&\leq& |\Lambda_{N,a}| Q\left( \sup_{{}^{0\leq s,\, t \leq N+1}_{\quad |s-t|<2}} |B_s - B_t|> x_N \right)
\end{eqnarray*}

\begin{eqnarray*}
&\leq& C N^{da} \sum^{N-2}_{i=0} Q\left( \sup_{i\leq t \leq i+3} B_t - \inf_{i\leq t \leq i+3} B_t > x_N \right)\\
&\leq& C N^{da+1} Q\left( \sup_{0\leq t \leq 3} |B_t| > \frac{x_N}{2} \right)\\
&\leq& C N^{da+1} Q\left(B_3 > \frac{x_N}{2} \right)\\
&\leq& C N^{da+1} e^{-C x^2_N}.
\end{eqnarray*}

\vspace{2ex}

\noindent With $x_N=\log N$ and recalling (\ref{Z-KMT}) from Step $1$, we see that $Q$-a.s., for $N$ large enough, 
\begin{eqnarray*}
 P_{N,a} \left( e^{\beta_{N,a} H(N,\alpha N^a)}\right) = P^c_{N,a} \left( e^{\beta_{N,a} Br(N,\alpha N^a)}\right) \times O(e^{\beta_{N,a} N^a \log N}).
\end{eqnarray*}

\noindent Again, this will imply that
\begin{eqnarray}\label{Z-cont}
 \log Z_{\beta_{N,a}}(N,\alpha N^a) = \log Z^{{\rm Br}}_{\beta_{N,a}}(N,\alpha N^a) + o(\beta_{N,a} N^{(1+a)/2}),
\end{eqnarray}

\subsubsection{Third Step: scaling}

Observe that, for a fixed path $\ns \in \Omega_{N,a}$,
\begin{eqnarray*}
\nBr(N,\alpha N^a)(\ns_{\cdot}) = \sqrt{N} \nBr(1,\alpha N^a)(\ns_{\cdot /N})
= N^{(1-a)/2} \nBr(N^a,\alpha N^a)(\ns_{\cdot \times N^{a-1}}),
\end{eqnarray*}

\noindent where the equalities hold in law. Note also that $\ns_{\cdot \times N^{a-1}}\in \Omega_{N^a,\alpha N^a}$.
It follows that

\begin{eqnarray*}
  Z^{{\rm Br}}_{\beta_{N , a}}(N,\alpha N^a) 
&=& P^c_{N,a}\left( \exp \beta_{N,a} {\bf Br}(N,\alpha N^a)(\ns) \right)\\
&=& P^c_{N^a,\alpha N^a}\left( \exp \beta_{N,a} N^{(1-a)/2} \nBr(N^a,\alpha N^a)(\ns_{\cdot \times N^{a-1}}) \right)\\
&=& P^c_{N^a,\alpha N^a}\left( \exp \beta \nBr(N^a,\alpha N^a)(\ns_{\cdot \times N^{a-1}}) \right).
\end{eqnarray*}

\noindent But the last expression is simply $Z^{{\bf Br}}_{\beta}(N^a,\alpha N^a)$ so that, by Theorem \ref{MO-th-d},

\begin{eqnarray}\label{Z-MO}
 \lim_{N\to +\infty} \frac{1}{N^a} \log Z^{{\rm Br}}_{\beta_{N,a}}(N,\alpha N^a) = p(\beta, \alpha, d).
\end{eqnarray}

\noindent From (\ref{Z-cont}) and (\ref{Z-MO}), we can deduce that the limit (\ref{ZMO-d})
holds in law.

\subsubsection{Final Step: concentration}
So far, we proved convergence in law for the original problem. But we can
write a convenient concentration inequality for the free energy with respect
to his average, in the Gaussian case. So, a.s. convergence holds for Gaussian,
and, according to step $1$, for any environment.

The classical concentration inequality for Gaussian random variables can be stated as
follows:

\begin{theorem}
 Consider the standard normal distribution $\mu$ on $\re^K$. If $f:\re^K \to \re$ is Lipschitz
 continuous with Lipschitz constant $L$, then
 
 \begin{eqnarray*}
 \mu \left( \,x: \, |f(\nx) - \int f d\mu| \geq u \right)\leq 2 \exp \{ - \frac{u^2}{2L^2} \}.
 \end{eqnarray*}
\end{theorem}

\noindent For a detailed exposition of concentration of measures, see for example, the lecture notes
of Ledoux \cite{Ledoux-SF}. Define 

\begin{eqnarray*}
 F(z) = \frac{1}{N^{a}} \log P_{N,a}\left(e^{\beta_{N,a} \sum^{N+dN^a}_{t=1} z(\nS_t)}\right).  
\end{eqnarray*}

\noindent It is easy to prove that $F$ is a Lipschitz continuous function with Lipschitz constant $CN^{-a/2}$. By Gaussian concentration,
this yields

\begin{eqnarray}\nonumber
 Q \left\lbrace \left| \frac{1}{N^{a}} \log Z_{\beta_{N,a}}(N,\alpha N^a) -
    \frac{1}{N^{a}} Q \log Z_{\beta_{N,a}}(N,\alpha N^a) \right| > u \right\rbrace \leq 2 \exp - \frac{N^a u^2}{2C^2}.\\
 \nonumber
 \, \\ 
\label{MO-finalconc}
\end{eqnarray}

\noindent This ends the proof of the theorem. \hfill $\square$

\vspace{2ex}


\subsection{Very asymmetric cases}\label{vac}

We now consider an even more asymmetric case: let $\na=(a_1,\cdots,a_d)$ with $0\leq a_i \leq a$ for all $i$ but $a_i=a$ for exactly $d-l$ 
values of $i$, $1\leq l <d$, and consider paths from the origin to points of type
$\alpha N^{\na}$ with coordinates $\alpha_i N^{a_i}$, $\alpha_i > 0$.
\begin{theorem}
Let $\alpha'$ be the vector of $\re^{d-l}$ which coordinates are those of $\alpha$ for the indexes $i$
such that $a_i = a$. Then,

\begin{eqnarray*}
  \lim_{N\to +\infty} \frac{1}{\beta_{N,a}N^{(1+a)/2}} \log Z_{\beta_{N,a}} (N, \alpha N^{\na}) = p(\beta, \alpha', d-l)/\beta.
\end{eqnarray*}
\end{theorem}

\begin{proof}
 The proof is very similar to the proof of Proposition \ref{prop-cont-PTP}. We will consider the simple case
$d=2$ and a final point of type $(N^a, N^b)$ with $b<a$. We then have to prove convergence to $p(\beta,1,1)/\beta$.
The general case follows easily. We can think of $h$ as $h=h_N=N^{(b-a)}$.

\noindent From the proof of Theorem \ref{MOregime-d}, we have to remember that

\begin{eqnarray*}
 \log \oZ_{\beta_{N,a}}(N,\alpha N^a) =  \log Z^{\nBr}_{\beta}(N^a,\alpha N^a) +  o(N^a)
\end{eqnarray*}

\noindent  Denote by $Z(N,M,L)$ (resp. $\oZ(N,M,L)$) the normalized (resp. non-normalized) partition function over discrete paths from the origin 
to $(N,M,L)$. We perform the same decomposition than before:

\begin{eqnarray*}
 &&Z_{\beta_{N,a}}(N, N^a, N^b) = \frac{\oZ_{\beta_{N,a}}(N,N^a, N^b)}{\oZ_0(N,N^a, N^b)}\\
&\ & \quad  = \frac{1}{\oZ_0(N,N^a, N^b)}\sum_{(\nm_i)_i} \Pi_i \oZ(i; \nm_i,\nm_{i+1})
\end{eqnarray*}

\noindent Here, $0\leq i \leq N^b-1$ and $\nm_{N^b}=N^a$.
Recalling Remark \ref{rem-phi}, the cardinality of the set $J_N$ of the possible configurations of $(\nm_i)$
satisfies $|J_N| \sim \exp \{ cN^{(a+b)/2} \log N \}$. 
For a fixed $\nm_i$,
recalling that the environment variables are centered,

\begin{eqnarray}\nonumber
 &&Q \left( \log \Pi_i \oZ_{\beta_{N,a}}(i; \nm_i,\nm_{i+1}) \right) 
=  Q\left(\log \Pi_i \oZ_{\beta_{N,a}}(0; \nm_i,\nm_{i+1})\right)\\
 \nonumber
 &\ & \quad \quad = \log \oZ_0(N,N^a,0) 
+ Q\left(\log \frac{\Pi_i \oZ_{\beta_{N,a}}(0;\nm_i,\nm_{i+1})}{\oZ_0(N,N^a,0)}\right)\\
 \nonumber
&\ & \quad \quad \leq \log \oZ_0(N,N^a,0)  + Q \left( \log Z_{\beta_{N,a}}(N,N^a,0)  \right)\\
 \nonumber
&\ & \quad \quad \leq \log \oZ_0(N,N^a,0)  + Q \left( \log Z^{\nBr}_{\beta}(N^a,N^a)\right) +  o(N^a)  \\
 \nonumber
&\ & \quad \quad \leq \log \oZ_0(N,N^a,0) + N^a \, p(\beta,1,1) + o(N^a). 
\end{eqnarray}

\noindent Now, if

\begin{eqnarray*}
\log Z_{\beta_{N,a}}(N,N^a,N^b) >  N^a \left( p(\beta,1,1) + \epsilon \right),
\end{eqnarray*}

\noindent there must exist some $(\nm_i)_i$ such that

\begin{eqnarray*}
\log \Pi_i \oZ_{\beta_{N,a}}(i; \nm_i,\nm_{i+1}) >  \log \oZ_0(N, N^a,N^b) + N^a \left( p(\beta,1,1) + \epsilon \right)- \log |J_N|.
\end{eqnarray*}

\noindent Using the fact that $\oZ_0(N, N^a,N^b) >  \oZ_0(N, N^a,0)$,  (\ref{MO-finalconc}) and the union bound, we find that

\begin{eqnarray*}
Q \left( Z_{\beta_{N,a}}(N,N^a,N^b) >  N^a \left( p(\beta,1,1) + \epsilon \right) \right) \leq |J_N| \exp \{ -c \epsilon^2 N^a\},
\end{eqnarray*}

\noindent for $\epsilon$ small enough. As $\log |J_N| = o(N^a)$, the RHS of the last display is summable.
The result follows by Borel-Cantelli.
\end{proof}

 


\section{One-dimensional directed polymers with a huge drift}\label{phd}

We now turn to the study of directed polymers with a drift growing with $N$.
This section contains the proofs of Theorems \ref{strongdrift} and \ref{fluct23}.

\subsection{The free energy}\label{sdfe}
Let us first sketch the proof of Theorem \ref{strongdrift}: we parametrize the terminal points conveniently:

\begin{eqnarray*}
 N=n(1+u).
\end{eqnarray*}

\noindent Thus $n=N/(1+u)$ and $N-n=Nu/(1+u)$. We can then rewrite (\ref{Zdrift}) as

\begin{eqnarray}\label{Zdrift2}
 Z^{(h)}_N = \sum_{u} \oZ_{\beta} \left( \frac{N}{1+u},\frac{Nu}{1+u}\right) \exp \left\lbrace -\gamma N^{(1-a)/2} \times \frac{Nu}{(1+u)}\right\rbrace.
\end{eqnarray}
 
\noindent Now, for $u$ in an interval $I_N=[N^{\kappa_0},N^{\kappa_1}]$, we will have

\begin{eqnarray}\label{unif}
 \oZ_{\beta} \left( \frac{N}{1+u},\frac{Nu}{1+u}\right) = \exp \left\lbrace
 2\beta \frac{N \sqrt{u}}{1+u} + o(1), 
\right\rbrace
\end{eqnarray}

\noindent uniformly in $u$. Then, 

\begin{eqnarray}\label{Zdrift3}
 Z^{(h)}_{\beta,N} \sim \sum_{u \in I_N}  \exp N \left\lbrace 2\beta \frac{ \sqrt{u}}{1+u} -\gamma N^{(1-a)/2} \times \frac{u}{(1+u)}\right\rbrace.
\end{eqnarray}

\noindent Define the function

\begin{eqnarray*}
 f_N(u)=2\beta \frac{ \sqrt{u}}{1+u} -\gamma N^{(1-a)/2} \times \frac{u}{(1+u)}.
\end{eqnarray*}

\noindent It attains its global maximum at a point $u^{*}_N \sim (\beta^2/ \gamma^2) N^{a-1}$ (in short, we will omit the 
dependence in $N$), with $f_N(u^*)\sim \beta^2 N^{(a-1)/2}/\gamma$. So, by
Laplace method, we will have

\begin{eqnarray*}
 Z^{(h)}_{\beta,N} = \exp \left\lbrace N  f(u^*_N) + o(1) \right\rbrace = \exp \left\lbrace N^{(1+a)/2} \beta^2/\gamma + o(1)\right\rbrace,
\end{eqnarray*}

\noindent which would finish the proof.

\begin{remark}
 The proof is split in three steps. The first one gives the lower bound in the Theorem, minoring the whole sum by one term, given by a $u$ very close to the minimizer. This is the easy part.

\noindent The second step will consist mainly in proving the uniformity in (\ref{unif}) (but replacing $=$ by $\leq$). This will be done by applying uniformly the KMT approximation in the whole interval $I_N$, and then applying some deviation inequality for the Brownian percolation.

\noindent The third step will be to prove that the $u$'s outside $I_N$ do not contribute to the sum.
\end{remark}

\begin{proof}[Proof of Theorem \ref{strongdrift}:]
 \noindent \textsc{First step:} We will now provide the lower bound: recall the notation in (\ref{Zdrift2}) and
  observe that for the value $u^*$, the asymptotics of $n$ and $N-n$ fit the situation studied in (\ref{thermo-thbox}). An easy computation yields:

 \begin{eqnarray*}
  &\ &\liminf_{N\to +\infty} \frac{1}{N^{(1+a)/2}} \log Z^{(h)}_{\beta,N} \\
  &\geq& \lim_{N\to +\infty} \frac{1}{N^{(1+a)/2}} \log \oZ_{\beta} \left( \frac{N}{1+u^*},\frac{Nu^*}{1+u^*}\right) \exp \left\lbrace -\gamma N^{(1-a)/2} \times \frac{Nu^*}{(1+u^*)}\right\rbrace \\
  &=& \frac{\beta^2}{\gamma}.
 \end{eqnarray*}

\vspace{2ex}

\textsc{Second step:} Let $\delta>0$ and take $\kappa_1= (a-1)/2 - \delta$ in order to define $I_N=[N^{\kappa_0},N^{\kappa_1}]$. Here, $\kappa_0 > -1$ is introduced to discard small values of $u$ that have to be treated separately. Note that, in this interval, $N-n \sim Nu \leq N^{1+\kappa_1} = o(N^{(a+1)/2})$. 

We first couple the environment variables $\{\eta(t,x): 1\leq t \leq N,\, 1\leq x \leq N^{\kappa_1}\}$ row by row with Brownian motions as in the proof of Theorem \ref{MO-th}. This yields

\begin{eqnarray*}\label{KMTunif}
  \oZ_{\beta} \left( \frac{N}{1+u},\frac{Nu}{1+u}\right) &=& \oZ^{\nBr}_{\beta} \left( \frac{N}{1+u},\frac{Nu}{1+u}\right) \times
  O(e^{N^{1+\kappa_1} \log N})\\
  &\leq& \exp \left\lbrace \beta L\left( \frac{N}{1+u},\frac{Nu}{1+u}\right) \right\rbrace \times
  O(e^{N^{1+\kappa_1} \log N}), 
\end{eqnarray*}

\noindent uniformly for $u\in I_N$, where $\oZ^{{\bf Br}}_{\beta}(N,M)$ denotes the unnormalized partition function of the Brownian model (In the following, we only make use of the domination by $L(\cdot, \cdot)$ which can also be guessed directly from the results in \cite{BM}). 
Note that (\ref{thermo-thbox}) holds for $\oZ^{{\bf Br}}_{\beta}(N,M)$ with $M=O(N^a)$, as $|\Omega^c_{N,M}|$ is small compared to $\sqrt{MN}$: 

\begin{eqnarray}\label{Ocont-bound}
 \log |\Omega^c_{N,M}|\sim \log \frac{N^{M}}{(M)!} = O(N^a \log N). 
\end{eqnarray}

We now search for a convenient upper bound for the (normalized) Brownian partition function:

\begin{eqnarray*}
 &\ &Q \left\lbrace Z^{{\bf Br}}_N \left( \frac{N}{1+u},\frac{Nu}{1+u}\right) > \exp \beta \frac{N \sqrt{u}}{1+u}(2+\epsilon_N) \right\rbrace \\
&\leq& Q \left\lbrace \max_{\omega}{\bf Br}\left( \frac{N}{1+u},\frac{Nu}{1+u}\right) >   \frac{N \sqrt{u}}{1+u}(2+\epsilon_N) \right\rbrace \\
&\leq& Q \left\lbrace \max_{\omega}{\bf Br}\left( 1,\frac{Nu}{1+u}\right) >   \sqrt{\frac{N u}{1+u}}(2+\epsilon_N) \right\rbrace \\
&\leq& C \exp \left\lbrace - \frac{1}{C}\frac{N\sqrt{u}}{1+u} \epsilon_N^{3/2} \right\rbrace.
\end{eqnarray*}

\noindent The last inequality follows from Ledoux \cite{Ledoux}, Section $2.1$ (see also Proposition \ref{non-asymptotics-L}). Taking $\epsilon_N=N^{-\theta}$ with $\theta>0$ 
small enough, and applying Borel-Cantelli, we conclude that, for $N$ large enough,

\begin{eqnarray*}
 Z^{{\bf Br}}_N \left( \frac{N}{1+u},\frac{Nu}{1+u}\right) \leq \exp \left\lbrace 2\beta \frac{N \sqrt{u}}{1+u} + o(1) \right\rbrace,
\end{eqnarray*}

\noindent for all $u\in I_N$. Now, thanks to (\ref{Ocont-bound}), this is still true with $\oZ^{{\bf Br}}$ instead of $Z^{{\bf Br}}$.
We then get

\begin{eqnarray*}
 \oZ_N \left( \frac{N}{1+u},\frac{Nu}{1+u}\right) \leq \exp \left\lbrace 2\beta \frac{N\sqrt{u}}{1+u} + o(1) \right\rbrace,
\end{eqnarray*}

\noindent uniformly for $u\in I_N$. Once the Third Step is achieved, this uniform bound and Laplace Method will finish the proof.

\vspace{2ex}

\textsc{Third Step:} We are now interested in values $u \leq N^{\kappa_0}$ and $u\geq N^{\kappa_1}$. Again we have to split the proof in three.

Let us first focus on small values of $u$. Recall that, in this region, by the KMT coupling, we can work directly with Gaussians. Take $\theta' > 0$.

\begin{eqnarray*}
&&Q \left\lbrace \oZ^g \left( \frac{N}{1+u},\frac{Nu}{1+u}\right) > e^{\beta N^{\theta'}} \right\rbrace 
\leq Q \left\lbrace T^g \left( \frac{N}{1+u},\frac{Nu}{1+u}\right) >  N^{\theta'} \right\rbrace \\
&\ & \quad \quad \quad \leq Q \left\lbrace  \exists \ns \in \Omega_{\frac{N}{1+u},\frac{Nu}{1+u}}: H(\ns) >  N^{\theta'} \right\rbrace \\
&\ & \quad \quad \quad \leq |\Omega_{\frac{N}{1+u},\frac{Nu}{1+u}}| \exp \{ - N^{2\theta'-1} \} \\
&\ & \quad \quad \quad \leq \exp \{ cN^{(1+\kappa_0) \log N} - N^{2\theta'-1} \}.
\end{eqnarray*}

\noindent So, choosing $\kappa_0$ small enough and $1+\kappa_0/2 < \theta' < (1+a)/2$, we get, by Borel-Cantelli and by a computation analogous to (\ref{Ocont-bound}), that for $N$ large enough,

\begin{eqnarray*}
\oZ^g \left( \frac{N}{1+u},\frac{Nu}{1+u}\right) = o\left(e^{N^{(1+a)/2}}\right),
\end{eqnarray*}

\noindent for all $u\leq N^{\kappa_0}$.

For $N^{(a-1)/2 - \delta} \leq u \leq N^{(a-1)/2 + \delta}$, we have to
 couple the environment row by row with Gaussians until $N-n=N^{(1+a)/2 + \delta}
$ (just conserve the coupling already done in Step $2$ and add the missing rows).
 This will yield an error uniformly of order $N^{(1+a)/2 + \delta} \log N$. The
 point is that for $\delta$ small enough, the drift will be large compared with
 the point-to-point partition functions and the error in the approximation. In
 fact,

\begin{eqnarray*}
 h \times (N-n) \geq \gamma N^{1-\delta}.
\end{eqnarray*}

\noindent Recall that  we are working with Gaussians, denote $\Omega_{N,u}=\Omega_{\frac{N}{1+u},\frac{Nu}{1+u}}$,

\begin{eqnarray*}
 Q \left\lbrace \max_{\omega \in \Omega_{N,u}} H_N( \omega) > N^{\theta'} \right\rbrace
\leq \exp (1+a)N^{(1+a)/2 + \delta}\log N - N^{2\theta' -1},
\end{eqnarray*}

\noindent and, by Borel-Cantelli (taking, of course, $(1+a)/2 + \delta < 2\theta' -1$),

\begin{eqnarray*}
 \oZ^g_{\beta} \left( \frac{N}{1+u}, \frac{Nu}{1+u} \right) e^{h \times (N-n)}
\leq \exp \left\lbrace (1+a) N^{(1+a)/2 + \delta}\log N + \beta N^{\theta'} - \gamma N^{1-\delta}
\right\rbrace,
\end{eqnarray*}

\noindent where, as usual, the overline denotes that the partition function is unnormalized and the superscript $g$ stands for Gaussian environment. To insure that the drift is larger than the other terms, we have to take $\theta' < 1-\delta$ and $(1+a)/2 + \delta < 1-\delta$, both holding for $\delta < (1-a)/4$ and $\theta$ small enough. 
Now, this is also enough to neglect the error in the approximation as it is of order $N^{(1+a)/2 + \delta}$ 
too. The first condition we have encountered, namely $(1+a)/2 + \delta < 2\theta' -1$ is satisfied for $\delta < (1-a)/6$ and $\theta' < 1-\delta$, so that, choosing $\delta$ and $\theta$ according to these last restrictions gives that

\begin{eqnarray}\label{Zvanish}
 \oZ_{\beta} \left( \frac{N}{1+u}, \frac{Nu}{1+u} \right)e^{-h \times (N-n)} \to 0,
\end{eqnarray}

\noindent as $N\to +\infty$ uniformly for $N^{(1-a)/2 - \delta} \leq u \leq N^{(1-a)/2 + \delta}$. 

We are then left with the values $u >  N^{(a-1)/2 + \delta}$. This is an easy task: we can dominate each point-to-point partition function by the whole partition function (without drift!): 

\begin{eqnarray*}
 Z_N=Z_{\beta,N}= \sum_{\omega \in \Omega_N} e^{\beta H(\omega)},
\end{eqnarray*}

\noindent where $\Omega_N$ is the set of directed nearest-neighbor paths of length $N$. $Z_N$ grows at most as $e^{CN}$ for some constant $C> \lambda(\beta)+ \log 2d$, as we can see from

\begin{eqnarray*}
 Q(Z_N \geq e^CN) \leq e^{-CN} QZ_N = e^{(\lambda(\beta)+\log 2d -C) N}
\end{eqnarray*}

\noindent and Borel-Cantelli. Now, for the range of $u$'s we are considering, the drift satisfies, 

\begin{eqnarray*}
 h (N-n) > N^{1+\delta'},
\end{eqnarray*}

\noindent for large $N$, whenever $\delta' < \delta$, and then (\ref{Zvanish}) holds in this interval as well.

\end{proof}


\subsection{Moderate deviations for the partition function}

We now discuss the fluctuation of $\log Z^{(h_N)}_{\beta,N}$. For technical reasons, we have to restrict 
to $a<1/3$ for variable with finite exponential moments, and to $a<3/7$ for Gaussian variables 
(see Remark \ref{rem-KMT} at the end of this section).

We start proving the two following deviation inequalities:

\begin{theorem}\label{small-dev}
For all $a<3/7$, 
there exists a constant $C>0$ such that, 
for all $N\geq 1$ and $0 \leq \epsilon \leq N^{1-a}$,

\begin{eqnarray}\label{dev-up}
 Q \left\lbrace \log Z^{(h_N)}_{\beta,N} \geq \frac{\beta^2}{\gamma} N^{(1+a)/2}(1+\epsilon) \right\rbrace
 \leq C \exp \left\lbrace - \frac{N^a}{\!\! C} \epsilon^{3/2} \right\rbrace,
\end{eqnarray}

\noindent and for $a<1/3$, ($a<3/7$ for Gaussian disorder), $0\leq \epsilon \leq 1$,

\begin{eqnarray}\label{dev-down}
 Q \left\lbrace \log Z^{(h_N)}_{\beta,N} \leq \frac{\beta^2}{\gamma} N^{(1+a)/2}(1-\epsilon) \right\rbrace
 \leq C \exp \left\lbrace - \frac{\, N^{2a}}{\!\! C} \epsilon^3 \right\rbrace.
\end{eqnarray}
\end{theorem}

\noindent These are consequences of similar non-asymptotics deviation inequalities for the top eigenvalue of GUE random matrices that
we recall here in the context of Brownian percolation (see \cite{LR}, Theorem $1$ and \cite{Ledoux}, Chapter $2$, for a complete discussion of this 
topic):

\begin{proposition}\label{non-asymptotics-L}
There exists a constant $C_1>0$ such that, for all $N\geq 1$ and $\epsilon \geq 0$,

\begin{eqnarray}\label{L-dev-up}
 Q \left\lbrace L(1,N) \geq 2 \sqrt{N}(1+\epsilon) \right\rbrace
 \leq C_1 \exp \left\lbrace - \frac{N}{\!\! C_1} \epsilon^{3/2} \right\rbrace,
\end{eqnarray}

\noindent and for $0\leq \epsilon \leq 1$,

\begin{eqnarray}\label{L-dev-down}
 Q \left\lbrace L(1,N) \leq 2\sqrt{N}(1-\epsilon) \right\rbrace
 \leq C \exp \left\lbrace - \frac{N^{2}}{C} \epsilon^3 \right\rbrace.
\end{eqnarray}
\end{proposition}

\vspace{2ex}

\noindent We first transfer these inequalities to the LPP context:

\begin{proposition}\label{non-asymptotics-discrete}
For $M=O(N^a)$ with $a<3/7$, there exists a constant $C_a>0$ such that, for all $\epsilon \geq 0$,

\begin{eqnarray}\label{non-asymptotics-discrete-up}
Q \left\lbrace  T(N, M)\geq 2\sqrt{NM}(1+\epsilon) \right\rbrace \leq C_a \exp \{ - M \epsilon^{3/2} / C_a\}.
\end{eqnarray}

\noindent and, for $a<1/5$  ($a<1/2$ if the environment is Gaussian) and $0\leq \epsilon \leq 1$,

\begin{eqnarray}\label{non-asymptotics-discrete-down}
Q \left\lbrace  T(N, M)\leq 2\sqrt{NM}(1-\epsilon) \right\rbrace \leq C_a \exp \{ - M^2 \epsilon^3 / C_a\}.
\end{eqnarray}
\end{proposition}

\begin{remark}
 The result for Gaussian variables can be found in \cite{I}. The exponential case is covered in the Thesis of the same author. We present here a rather complete proof of the exponential case using KMT, both for completness and to state some inequalities that will be used in the following. 
\end{remark}

The core of the proof of Proposition \ref{non-asymptotics-discrete} consists in the following Lemma, whose proof is a simple application of the KMT coupling (Theorem \ref{KMT-th}):

\begin{lemma} There exist positive constants $C_2,\, C_3$ such that,

 \begin{eqnarray*}
  Q \left\lbrace \exp \{ |T(N, M) - T^g(N,M)| \} \right\rbrace
  \leq e^{C_2 M \log N},
 \end{eqnarray*}

 \begin{eqnarray*}
  Q \left\lbrace \exp \{ M^{-1}|L(N, M) - T^g(N,M)|^2 \} \right\rbrace
  \leq e^{C_3 M \log N},
 \end{eqnarray*}

\noindent where $T^g(N,M)$ is the discrete Gaussian last passage functional given by the KMT approximation. 
\end{lemma}

\begin{proof}

By the KMT approximation,

\begin{eqnarray*}
 Q \left\lbrace e^{\{ |T(N, M) - T^g(N,M)| \} } \right\rbrace &\leq&
 Q\left\lbrace \exp \{ 2 \sum^M_{i=0}  \sup_{j \leq N}|\sum^{j}_{k=0} \eta(k,i)-B^{(i)}_j|\} \right\rbrace \\
 &\leq& \left( 
 Q\left\lbrace \exp \{ 2 \sup_{j \leq N}|\sum^{j}_{k=0} \eta(k,1)-B^{(1)}_j|\} \right\rbrace 
 \right)^M \\
 &\leq& e^{C_ 1 M \log N}.
\end{eqnarray*}

For the second affirmation, remember that in the second step of the proof of Theorem \ref{MOregime}, we noticed that

\begin{eqnarray}\label{reflexion}
 Q\left\lbrace \sup_{{}^{0\leq s,t \leq N+1}_{|s-t|<2}} |B_s - B_t| > x \right\rbrace \leq C_4 e^{-x^2/C_4},
\end{eqnarray}

\noindent for some $C_4>0$. Then,

 \begin{eqnarray*}
  Q \left\lbrace \exp \{ M^{-1}|L(N, M) - T^g(N,M)|^2 \} \right\rbrace
  &\leq& \left( \exp \{ 2 (\sup_{{}^{1\leq s,t \leq N}_{|s-t|<2}}|B^{(1)}_s - B^{(1)}_t|)^2\}
  \right)^M \\
&\leq& e^{C_3 M \log N}
 \end{eqnarray*}

\noindent where we use (\ref{reflexion}) in the last step.
\end{proof}

The following Corollary is now straightforward:

\begin{cor}\label{cor-markov}
 For any sequence $(\theta_N)_N$ with $\theta_N >0$, we have
 
\begin{eqnarray}\label{ttg-markov}
 Q\left\lbrace |T(N,M) - T^g(N,M)| \geq \theta_N \right\rbrace \leq e^{-\theta_N + C_2 M\log N},
\end{eqnarray}

\begin{eqnarray}\label{ltg-markov}
 Q\left\lbrace |L(N,M) - T^g(N,M)| \geq \theta_N \right\rbrace \leq e^{-M^{-1}\theta^2_N + C_3 M\log N}.
\end{eqnarray}

\end{cor}

\textsc{Proof of Proposition \ref{non-asymptotics-discrete}}

 Let us prove (\ref{non-asymptotics-discrete-up}) ((\ref{non-asymptotics-discrete-down}) is proved
following the same lines). Remember $T^g(N,M)$ denotes the last passage percolation functional for the Gaussian environment given by the KMT coupling. Then,

 \begin{eqnarray*}
  &\ &Q \left\lbrace  T(N, M)\geq 2\sqrt{NM}(1-\epsilon) \right\rbrace 
  \leq Q \left\lbrace  L(N, M)\geq 2\sqrt{NM}(1-\epsilon/2) \right\rbrace \\
  &\ & \quad \quad \quad \quad \quad \quad+ \quad Q \left\lbrace  |T(N,M)-T^g(N, M)|\geq \frac{\epsilon}{2} \sqrt{NM} \right\rbrace \\
  &\ & \quad \quad \quad \quad \quad \quad+ \quad Q \left\lbrace  |L(N,M)-T^g(N, M)|\geq \frac{\epsilon}{2} \sqrt{NM} \right\rbrace \\
 \end{eqnarray*}
 
\noindent The first term can be treated using (\ref{L-dev-up}) and Brownian scaling: 

\begin{eqnarray*}
  Q \left\lbrace  L(N, M)\geq 2\sqrt{NM}(1-\epsilon/2) \right\rbrace \leq C_1 e^{- \frac{N^a}{C_1} \epsilon^{3/2}}.
\end{eqnarray*}

\noindent The remaining term can be treated with Corollary \ref{cor-markov} taking $\theta_N = \epsilon/2 \sqrt{NM}$.

\noindent The carefull analysis performed in \cite{I}, Section $5$, allows us to choose the uniform constant in (\ref{non-asymptotics-discrete-up}).
\hfill $\square$

\vspace{2ex}


\noindent We turn now to the proof of Theorem \ref{small-dev}.

\noindent  \textsc{Proof of the inequality \ref{dev-down}:} This follows by lowering the
partition function by one term: recall that $u^* \sim \beta^2/\gamma N^{a-1}$, and define
$n^* = N/(1+u^*)$. Then,

\begin{eqnarray*}
 &\ &Q \left\lbrace \log Z^{(h_N)}_{\beta,N} \leq \frac{\beta^2}{\gamma} N^{(1+a)/2}(1-\epsilon) \right\rbrace \\
 &\leq& Q \left\lbrace \beta T(n^*, N-n^*)-h_N \times (N-n^*)\leq \frac{\beta^2}{\gamma} N^{(1+a)/2}(1-\epsilon) \right\rbrace,
\end{eqnarray*}

\noindent Observe that

$$ h_N \times (N-n^*) = \frac{\beta^2}{\gamma} N^{(1+a)/2}. $$

\noindent We are then reduced to estimate the quantity

\begin{eqnarray}\label{dev-down-bound}
Q \left\lbrace  T(n^*, N-n^*)\leq \frac{2\beta}{\gamma} N^{(1+a)/2}(1-\epsilon/2) \right\rbrace.
\end{eqnarray}

\noindent which can be handled with (\ref{non-asymptotics-discrete-down}). \qed

\vspace{2ex}

\noindent  \textsc{Proof of the inequality \ref{dev-up}:} This proof is more involved as it requires to control all the terms 
in the sum defining $Z^{(h_N)}_{\beta,N}$. Again, we need to give a special treatment to the terms for which $N-n$ is not of the relevant order (namely $O(N^a)$).
We use the convenient parametrization $N-n= vN^a$ for some $v\geq 0$, i.e., $u = v N^{a-1}$. To lighten notation, let us denote 

\begin{eqnarray*}
 q(\epsilon,v) = Q\left\lbrace  \beta T(N, vN^a) - \gamma v N^{(1+a)/2}\geq \frac{\beta^2}{\gamma} N^{(1+a)/2}(1+\epsilon) \right\rbrace.
\end{eqnarray*}

Several cases have to be analyzed separately:

\vspace{2ex}

\noindent \textsc{Case $v\leq \beta^2/(2\gamma)^2$:} We use the fact that, for these values of $v$, $T(N, vN^a)$ is stochastically dominated
by $T(N, \beta^2/(2\gamma)^2 N^a)$. Then, neglecting the term $\gamma v$,

\begin{eqnarray*}
 q(v,\epsilon) &\leq& Q\left\lbrace  T(N, \beta^2/(2\gamma)^2 N^a) \geq \frac{\beta}{\gamma} N^{(1+a)/2}(1+\epsilon) \right\rbrace\\
  &\leq& C_a \exp \left\lbrace - \frac{\beta^2 N^a}{4 \gamma^2 C_a} \epsilon^{3/2} \right\rbrace. 
\end{eqnarray*} 

\vspace{2ex}

\noindent \textsc{Case $\beta^2/(2\gamma)^2 \leq v \leq 16\beta^2/\gamma^2$:} We make use of the fact that
$2\beta \sqrt{v} - \gamma v \leq \beta^2/\gamma$ for all $v \geq 0$ and that, for these values of $v$,
we have $1/\sqrt{v} \geq \gamma/(16 \beta)$. Then,

\begin{eqnarray*}
 q(\epsilon, v) &\leq& Q\left\lbrace  \beta T(N, vN^a)\geq 2\beta \sqrt{v}N^{(1+a)/2}+\frac{\beta^2 \epsilon}{\gamma} N^{(1+a)/2} \right\rbrace \\
 &\leq& Q\left\lbrace  \beta T(N, vN^a)\geq 2\beta \sqrt{v}N^{(1+a)/2} \left( 1 + \frac{\beta \epsilon}{2\gamma \sqrt{v}} \right) \right\rbrace \\
 &\leq& Q\left\lbrace  \beta T(N, vN^a)\geq 2\beta \sqrt{v}N^{(1+a)/2} \left( 1 + \frac{\epsilon}{32} \right) \right\rbrace \\
&\leq& C_a \exp \left\lbrace - \sqrt{v} \frac{N^a}{\!\!\! C_a} \epsilon^{3/2} \right\rbrace \\
&\leq& C_a \exp \left\lbrace - \frac{\beta^2 N^a}{\!\!\! \gamma^2 C'_a} \epsilon^{3/2} \right\rbrace,
\end{eqnarray*}

\noindent thanks to the lower bound we assumed on $v$. So far, we made no additional assumption on $\epsilon$.

\vspace{2ex}

\noindent \textsc{Case $16\beta^2 / \gamma^2 \leq v$:}

\begin{eqnarray}\nonumber
 q(\epsilon, v) &\leq& Q\left\lbrace  \beta L(N, vN^a) - \frac{\gamma v}{2} N^{(1+a)/2}\geq \frac{\beta^2}{2\gamma} N^{(1+a)/2}(1+\epsilon) \right \rbrace \\
\label{ttg-term}
&\ & \quad + Q\left\lbrace \beta |T(N,M) - T^g(N,M)| - \frac{\gamma v}{4} N^{(1+a)/2}\geq \frac{\beta^2}{4\gamma} N^{(1+a)/2}(1+\epsilon) \right\rbrace \\
\label{ltg-term}
&\ & \quad + Q\left\lbrace \beta |L(N,M) - T^g(N,M)| - \frac{\gamma v}{4} N^{(1+a)/2}\geq \frac{\beta^2}{4\gamma} N^{(1+a)/2}(1+\epsilon) \right\rbrace
\end{eqnarray}

\noindent Let us treat the first summand: assume that $K\beta^2/\gamma^2 < v \leq (K+1)\beta^2/\gamma$, for some $K\geq 16$.
Then,

\begin{eqnarray*}
 \frac{\beta^2}{2\gamma}(1+\epsilon) + K \frac{\beta^2}{2\gamma} \geq \frac{2\beta^2}{\gamma} \sqrt{K+1}
 \left(1 + \frac{\epsilon}{4\sqrt{K+1}}\right).
\end{eqnarray*}

\noindent So, recalling that $L(N, (K+1) \beta^2 N^a/ \gamma^2)$ stochastically dominates $L(N,vN^a)$ for these values of $v$,

\begin{eqnarray*}
 &\ &Q\left\lbrace \beta L(N,vN^a) \geq \left( \frac{\beta^2}{2\gamma}(1+\epsilon) + K \frac{\beta^2}{2\gamma}\right) N^{(1+a)/2} \right\rbrace \\
&\leq& Q\left\lbrace L \left( N, (K+1) \frac{2\beta}{\gamma^2} N^a\right) \geq
\frac{2\beta}{\gamma} \sqrt{K+1} N^{(1+a)/2} \left( 1 + \frac{\epsilon}{4 \sqrt{K+1}}\right) 
\right\rbrace \\
&\leq& C_1 \exp \left\lbrace -\frac{1}{C_1} (K+1) \frac{\beta^2}{\gamma^2}N^a \left( \frac{\epsilon}{4\sqrt{K+1}}\right)^{3/2}\right\rbrace \\
&\leq& C_5 \exp \left\lbrace - \frac{1}{C_5} N^a \epsilon^{3/2}\right\rbrace.
\end{eqnarray*}

\noindent The remaining terms (\ref{ttg-term}) and (\ref{ltg-term}) can be handle with (\ref{ttg-markov}) and (\ref{ltg-markov}) respectively. We then get

\begin{eqnarray*}
 &\ & q(\epsilon, v) \leq C_5 \exp \left\lbrace - \frac{1}{C_5} N^a \epsilon^{3/2}\right\rbrace 
  +  \exp \left\lbrace - c_1 v N^{(1+a)/2} - c_2 N^{(1+a)/2}(1+ \epsilon) \right\rbrace \\
  &\ & \quad \quad \quad +  \exp \left\lbrace - c_3 v N - c_4 v^{-1}(1+\epsilon)^2 \right\rbrace \\
 &\ & \quad \leq  C_5 \exp \left\lbrace - \frac{1}{C_5} N^a \epsilon^{3/2}\right\rbrace 
  +  \exp \left\lbrace - c'_2 \epsilon N^{(1+a)/2} \right\rbrace \\
 &\ & \quad \quad \quad +  \exp \left\lbrace - c_3 v N - c'_4 \frac{\epsilon^2}{v} N \right\rbrace \\
 &\ & \quad \leq C_5 \exp \left\lbrace - \frac{1}{C_5} N^a \epsilon^{3/2}\right\rbrace 
 + c_5 \exp \left\lbrace - \frac{1}{c_5} \epsilon N^{(1+a)/2}\right\rbrace.
\end{eqnarray*}

\noindent Observe that, for $0\leq \epsilon \leq N^{1-a}$, there is a constant $C_6> 0$ such that

\begin{eqnarray*}
 q(\epsilon, v) \leq C_6 \exp \left\lbrace - \frac{1}{C_6} N^a \epsilon^{3/2}\right\rbrace,
\end{eqnarray*}

\noindent which ends the proof.

\hfill $\square$

Let us observe that, for $\epsilon > N^{1-a}$, there is a constant $C_7>0$ such that

\begin{eqnarray}\label{large-epsilon}
 q(\epsilon, v) \leq C_7  \exp \left\lbrace - \frac{1}{C_7} \epsilon N^{(1+a)/2}\right\rbrace.
\end{eqnarray}

\vspace{2ex}


\subsection{Fluctuation bounds}\label{sdfb}

We can now complete the proof of Theorem \ref{fluct23}. The argument to deduce the fluctuation bounds from our moderate deviations is very general and can be found in \cite{Ledoux} in the context of random matrices.
To lighten notations, let us denote

\begin{eqnarray*}
 X_N = \log Z^{(h_N)}_{\beta, N}, \quad \quad x_N = \frac{\beta^2}{\gamma}N^{(1+a)/2}.
\end{eqnarray*}
 The upper bound follows from the previous deviation inequalities by a direct computation:

\begin{eqnarray*}
 &\ & Q(X_N - x_N)^2 = \int^{+\infty}_0  Q\left\lbrace (X_N-x_N)^2 \geq t \right\rbrace dt\\
 &\ & \quad \quad \quad \leq \int^{+\infty}_0  Q\left\lbrace X_N-x_N \geq \sqrt{t} \right\rbrace dt +
 \int^{x_N}_0 Q\left\lbrace X_N-x_N \leq -\sqrt{t} \right\rbrace dt\\
 &\ & \quad \quad \quad = 2 N^{1+a} \int^{+\infty}_0  u\, Q\left\lbrace X_N \geq \frac{\beta^2}{\gamma} (1+u) \right\rbrace du \\
 &\ & \quad \quad \quad \quad + \quad
 2N^{1+a}\int^{\beta^2/\gamma}_0 u\, Q\left\lbrace X_N \leq \frac{\beta^2}{\gamma} (1-u) \right\rbrace du
\end{eqnarray*}

\noindent Let us bound the first integral. The second one can be treated in the same way. We apply (\ref{dev-up}) from
Theorem \ref{small-dev} and split the interval of integration:

\begin{eqnarray}\nonumber
 &\ &\int^{+\infty}_0  uQ\left\lbrace X_N \geq \frac{\beta^2}{\gamma} (1+u) \right\rbrace du \\
 \label{int-up}
 &\ & \quad \quad \quad
 = C \int^{1}_0 u e^{- \frac{1}{C} N^a u^{3/2}} du \, + \, C_7 \int^{+\infty}_1 u e^{- \frac{1}{C_7}N^{\frac{(1+a)}{2}} u} du.
\end{eqnarray}

\noindent The second integral in this last display is easily seen to decrease as $\exp \{- N^{(1+a)/2}\}$. For the first integral,
observe that the integrand is $O(N^{-2a/3})$ in $[0,N^{-2a/3}]$ and decreases exponentially fast outside this interval. 
Then,

\begin{eqnarray*}
 \int^1_0 u e^{- \frac{1}{C} N^a u^{3/2}} du \leq C' N^{-4a/3},
\end{eqnarray*}

\noindent for some $C'>0$. Putting this back into (\ref{int-up}), we found

\begin{eqnarray*}
 \int^{+\infty}_0  uQ\left\lbrace X_N \geq \frac{\beta^2}{\gamma} (1+u) \right\rbrace du \leq C e^{N^{1-a/3}}.
\end{eqnarray*}

\noindent As we already mentioned, the deviations on the left of the mean can be treated similarly.
This gives the upper bound.
For the lower bound, observe that

$$ \beta T(n^*, N-n^*)-h_N \times (N-n^*) \sim \beta T(N, \frac{\beta^2}{\gamma^2}N^a) - \frac{2\beta^2}{\gamma}N^{(1+a)/2}.$$

\noindent Then, applying Jensen's inequality,

\begin{eqnarray*}
 &\ &Q \left\lbrace \left( \log Z^{(h_N)}_{\beta, N} - \frac{\beta^2}{\gamma}N^{(1+a)/2}\right)^2 \right\rbrace
 \geq \left( Q \left\lbrace \log Z^{(h_N)}_{\beta, N} - \frac{\beta^2}{\gamma}N^{(1+a)/2} \right\rbrace \right)^2\\
 &\ & \quad \quad \quad \quad \quad \quad \geq
 \left( Q \left\lbrace  \beta T(N, \beta^2/\gamma^2 N^a) - 2\beta^2/\gamma N^{(1+a)/2} 
      \right\rbrace \right)^2 \\
\end{eqnarray*}

\vspace{4ex}
\noindent Now, recall \cite{BM} that

$$ \frac{T(N, \beta^2/\gamma^2 N^a) - 2\beta^2/\gamma N^{(1+a)/2}}{N^{(\frac{1}{2}-\frac{a}{6})}}$$

\noindent converges in law to a Tracy-Widom. Then, recalling that the Tracy-Widom law has a strictly positive expected value,

$$\left( Q \left\lbrace  \beta T(N, \beta^2/\gamma^2 N^a) - 2\beta^2/\gamma N^{(1+a)/2}  \right\rbrace \right)^2 \geq cN^{1 - a/3},$$

\noindent for some $c> 0$. This ends the proof.
\hfill $\square$

\begin{remark}\label{rem-KMT}
 Again, the condition $a<1/5$ seems to be a technical limitation due to our use of the KMT approximation. For a more
extensive discussion on asymptotics and non-asymptotics small deviations for asymmetric last-passage percolation, see \cite{I}.
\end{remark}

\begin{remark}\label{remark-TW}
 The limit law of the properly centered and rescaled partition function  should be the GUE Tracy-Widom law from random matrix theory. 
A proof of this fact would need to refine the analysis performed in the first section of this chapter to reduce the relevant 
values of $u$'s to an interval $[cN^{1-a}-\epsilon_N,cN^{1-a}-\epsilon_N]$ with $c=\beta^2/\gamma^2$ and $\epsilon_N \to 0$ fast enough. 
This can be done without much effort, but, in order to identify the limit law as the Tracy-Widom, we also need
a joint control of expressions of the form $L(N-cN^a - s N^{2a/3}, cN^a + s N^{2a/3})$ for $s$ ranging over a large interval.
The result we are searching for can be expressed as follows: for $s\in \re$

$$ N^{1/6}\left\lbrace L(N-sN^{2/3},N+sN^{2/3}) - 2N \right\rbrace \to {\rm Ai}(s)-s^2$$

\noindent where ${\rm Ai}(\cdot)$ is a continuous version the Airy process. This is a stationary process which marginals
are the Tracy-Widom law. See \cite{J-PNG} for a related result and a precise
description of the Airy process.
\end{remark}

\end{document}